\documentclass[11pt]{amsart}

\usepackage{amssymb}
\usepackage{amsmath}
\usepackage{amsthm}
\usepackage{amscd}
\usepackage{amsfonts}
\usepackage{ascmac}
\usepackage[usenames]{color}
\usepackage{graphics}

\newtheorem{theorem}{Theorem}[section]

\newtheorem{corollary}[theorem]{Corollary}
\newtheorem{definition}[theorem]{Definition}

\renewcommand{\bar}{\overline}

\newcommand{\pa}{\partial}
\renewcommand{\phi}{\varphi}

\newcommand{\ka}{K\"ahler }

\newcommand{\C}{{\mathbb C}}

\newcommand{\ke}{K\"ahler-Einstein }

\newcommand{\ga}{\alpha}
\newcommand{\gb}{\beta}

\newcommand{\CC}{\mathbb{C}}
 
\newcommand{\barj}{{\overline j}}

\newcommand{\barz}{{\overline z}}

\newcommand{\barpartial}{{\overline \partial}}

\newcommand{\ZZ}{\mathbb{Z}}

\newcommand{\mapright}[1]{\smash{\mathop{   \hbox to 0.7cm{\rightarrowfill}}
  \limits^{#1}}}

\newcommand{\Ric}{\operatorname{Ric}}

\newcommand{\grad}{\mathrm{grad}}
\newcommand{\Fut}{\mathrm{Fut}}
\newcommand{\Vol}{\mathrm{Vol}}

\title{Residue formula for an obstruction to Coupled K\"ahler-Einstein metrics}
\author{Akito Futaki and Yingying Zhang}
\address{Yau Mathematical Sciences Center, Tsinghua University, Haidian district, Beijing 100084, China}
\email{futaki@tsinghua.edu.cn}
\address{Yau Mathematical Sciences Center, Tsinghua University, Haidian district, Beijing 100084, China}
\email{yingyzhang@tsinghua.edu.cn}
\date{October 13th, 2019}

\begin{document}

\begin{abstract}
We obtain a residue formula for an obstruction to the existence of coupled K\"ahler-Einstein metrics
described in \cite{futakizhang}. We apply it to an example studied by the first author \cite{futaki83.1} and Hultgren \cite{Hultgren17} 
which is a Fano manifold with reductive automorphism, does not admit a K\"ahler-Einstein metric but still admits coupled K\"ahler-Einstein
metrics. 
\end{abstract}

\maketitle

%\footnotetext[1]{The work of the first author was partially supported by .}
\footnotetext[1]{The work of the second author was partially supported by Tsinghua University Initiative Scientific Research Program.}

\section{Introduction}
A $k$-tuple of K\"ahler metrics $\omega_1,\ \cdots, \omega_k$ on a compact K\"ahler manifold $M$
is called coupled K\"ahler metrics 
 if it satisfies
\begin{equation}\label{coupled KE}
\Ric(\omega_1) = \ \cdots = \Ric(\omega_k) = \lambda \sum_{\ga=1}^k \omega_\ga
\end{equation}
for $\lambda = -1,\ 0$ or $1$ where $\Ric(\omega_\ga)$ is the Ricci form of $\omega_\ga$ (we do not distinguish K\"ahler metrics $g_\ga$ and their K\"ahler forms $\omega_\ga$).
Such metrics were introduced  by Hultgren and Witt Nystr\"om \cite{HultgrenWittNystrom18}. 
If $\lambda = 0$ this is just a $k$-tuple of Ricci-flat metrics and the existence is well-known for compact K\"ahler manifolds
with $c_1(M) = 0$ by the celebrated solution by Yau 
\cite{yau78} of the Calabi conjecture. For $\lambda = -1$ or $\lambda = 1$ the existence problem is an extension for the problem
for negative or positive K\"ahler-Einstein metrics, and an obvious condition is $c_1(M) < 0$ or $c_1(M) > 0$. 
Hultgren and Witt Nystr\"om \cite{HultgrenWittNystrom18} proved the existence of the solution for $\lambda = -1$ under the condition
$c_1(M) < 0$ extending \cite{yau78} and \cite{aubin76}, and there are many interesting results for $\lambda = 1$ under the condition $c_1(M) > 0$
including attempts to extend \cite{CDS3} and \cite{Tian12}. Further studies of coupled K\"ahler-Einstein
metrics have been done in \cite{pingali}, \cite{Hultgren17}, \cite{futakizhang}, \cite{DelcroixHultgren1812}, \cite{Takahashi1901}, \cite{DatarPingali1901},
\cite{Takahashi1904}, \cite{Nakamura19}.

In this paper we derive a residue formula for an obstruction to the existence of positive coupled K\"ahler-Einstein metrics described in
our previous paper \cite{futakizhang} 
and apply to a computation of an example which appeared in Hultgren \cite{Hultgren17}. 

The obstruction is described as follows. 
Let $M$ be a Fano manifold of complex dimension $m$. Assume the anticanonical line bundle has a 
splitting $K_M^{-1}=L_1\otimes\cdots\otimes L_k$ into the tensor product of ample line bundles $L_\ga\to M$.
Then we have  $c_1(L_\ga)=\frac{1}{2\pi}[\omega_\ga]$ for a K\"ahler form $\omega_\ga=\sqrt{-1}g_{\ga\, i\bar j} dz^i\wedge d\bar z^j$, and thus

\begin{equation*}
c_1(M)=\frac{1}{2\pi}\sum\limits_{\ga=1}^k[\omega_\ga].
\end{equation*}
For each $\omega_\ga$ we have $f_\ga\in C^\infty(M)$ such that 
\begin{equation*}
\Ric(\omega_\ga)=\sum\limits_{\gb=1}^k\omega_\gb+\sqrt{-1}\pa\bar\pa f_\ga,
\end{equation*}
where $f_\ga$ are normalized by 
\begin{equation}\label{normalization}
e^{f_1}\omega_1=\cdots=e^{f_k}\omega_k.
\end{equation}
Note that this normalization still leaves an ambiguity up to a constant. But we ignore this ambiguity since
it does not cause any problem in later arguments. Of course $\omega_1,\ \cdots,\ \omega_k$ are coupled K\"ahler-Einstein metrics
if and only if $f_\ga$ are all constant.

Let $X$ be a holomorphic vector field. Since a Fano manifold is simply connected there exist complex-valued smooth functions defined up to constant
$u_\ga$ such that 
\begin{equation}\label{Hamiltonian}
i_X\omega_\ga=\bar\pa(\sqrt{-1}u_\ga).
\end{equation}
By the abuse of terminology we call $u_\ga$ the Hamiltonian function of $X$ with respect to $\omega_\ga$ though
$u_\ga$ is a Hamiltonian function for the imaginary part of $X$ in the usual sense of symplectic geometry only when 
$u_\ga$ is real valued.
In Theorem 3.3 of \cite{futakizhang}, it is shown for some choices of $u_\ga$ we have
\begin{equation}\label{Th3.3}
\Delta_{\ga} u_\ga+(\grad_{\ga} u_\ga)f_\ga = -\sum\limits_{\gb=1}^k u_\gb.
\end{equation}
where $\Delta_{\ga}=-\bar\pa^*_{ \ga}\bar\pa$ is the Laplacian with respect to $\omega_\ga$
and $\grad_{\alpha}u_\ga$ 
is the type $(1,0)$-part of the gradient of $u_\ga$ expressed as 
$\grad_{\alpha}u_\ga = g_{\alpha}^{i\barj}\frac{\partial u_\ga}{\partial \barz^j} \frac{\partial}{\partial z^i}$  in terms of local holomorphic coordinates $( z^1, \dots, z^m)$. The case of $k=1$ of this result has been obtained in \cite{futaki87}.
If we replace $u_\ga$ by $u_\ga^{c_\ga}=u_\ga+c_\ga$ the equations \eqref{Th3.3} are satisfied for $u_\ga^{c_\ga}$ 
if and only if 
\begin{equation}\label{ambiguity} 
\sum\limits_{\ga=1}^k c_\ga=0.
\end{equation}
\begin{definition}[\cite{futakizhang}]
With the choice of $u_\ga$ satisfying \eqref{Th3.3}
the Lie algebra character is defined as
\begin{equation}\label{futaki}
\begin{split}
\Fut: \mathfrak h(M)&\to \C\\
X\quad&\mapsto \Fut(X)=\sum\limits_{\ga=1}^k\frac{\int_M u_\ga\ \omega_\ga^m}{\int_M \omega_\ga^m}.
\end{split}
\end{equation}
\end{definition}
\noindent
Notice that this definition of $\Fut$ is not affected by the ambiguity of the choice of $u_\ga$ because of
\eqref{ambiguity}. Note also $\Fut$ is the coupled infinitesimal form of the group character obtained in \cite{futaki86}.

To formulate the localization formula let $Z=\bigcup\limits_{\lambda\in\Lambda}Z_\lambda$ be zero sets of $X$ where $Z_\lambda$'s are connected components. Let $N_\ga(Z_\lambda)=(T_M|_{Z_\lambda})/{T_{Z_\lambda}}$ be the normal bundle of $Z_\lambda$ with respect to $\omega_\ga$. Then the Levi-Civita connection $\nabla^\ga$ of $\omega_\ga$ naturally induces an endomorphism $L^{N_\ga}(X)$ of $N_\ga(Z_\lambda)$ by 
\[L^{N_\ga}(X)(Y)=(\nabla^\ga_Y X)^{\perp}\in N_\ga(Z_\lambda), \quad \text{ for any } Y\in N_\ga(Z_\lambda).\]
We also assume $Z$ is nondegenerate in the sense that $L^{N_\ga}$ is nondegenerate.
Let $K_\ga$ be the curvature of $N_\ga(Z_\lambda)$.
The localization formula of $\Fut(X)$ we obtain is the following.
\begin{theorem}\label{localization}
Let $M$ be a Fano manifold with $K_M^{-1}=L_1\otimes\dots\otimes L_k$. Let $X$ be a holomorphic vector field with nondegenerate zero sets $Z=\bigcup_{\lambda\in\Lambda}Z_\lambda$, then
\begin{equation}\label{locfut}
\Fut(X)=\frac{1}{m+1}\sum\limits_{\ga=1}^k\Bigg(\frac{\sum\limits_{\lambda\in \Lambda}\int_{Z_\lambda}\big(\big(E_\ga+c_1(L_\ga)\big)|_{Z_\lambda}\big)^{m+1}\big /\det\big((2\pi)^{-1}(L^{N_\ga}(X)+\sqrt{-1}K_\ga)\big)}{\sum\limits_{\lambda\in \Lambda}\int_{Z_\lambda}\big(\big(E_\ga+c_1(L_\ga)\big)|_{Z_\lambda}\big)^{m}\big /\det\big((2\pi)^{-1}(L^{N_\ga}(X)+\sqrt{-1}K_\ga)\big)}\Bigg).
\end{equation}
where $E_\ga\in \Gamma(End(L_\ga))$ is given by $E_\ga s = u_\ga s$ with $L^{N_\ga}$ and $K_\ga$ being as above.
\end{theorem}
\begin{corollary}
If $Z$ contains only discrete points, then
\begin{eqnarray*}
\Fut(X)&=&\frac{1}{m+1}\sum\limits_{\ga=1}^k\Bigg(\frac{\sum\limits_{p\in Z}(u_\ga(p))^{m+1}/\det(\nabla X)(p)}{\sum\limits_{p\in Z}(u_\ga(p))^{m}/\det(\nabla X)(p)}\Bigg)\\
&=&\frac{1}{m+1}\Bigg(\sum\limits_{\ga=1}^k\frac{\sum\limits_{p\in Z}(u_\ga(p))^{m+1}}{\sum\limits_{p\in Z}(u_\ga(p))^{m}}\Bigg).
\end{eqnarray*}
\end{corollary}

We can apply the obtained localization formula for the invariant $\Fut$ in the coupled situation to verify the example considered in Hultgren's paper \cite{Hultgren17}. This example was first considered by the first author in \cite{futaki83.1}, where he showed that the invariant $\Fut$ is non-vanishing, hence there does not exist a \ke metric on this example though the automorphism group is reductive and thus Matsushima's condition \cite{matsushima57} is satisfied. Later, in \cite{futakiMabuchiSakane90}, the localization formula in \cite{futakimorita85} was used to show a much simpler computation of the invariant $\Fut$ can be done. Hultgren \cite{Hultgren17} considered decompositions of the anticanonical line bundle, and proved in a special case of the decomposition there does exist coupled \ke metric on this manifold. 

The rest of the paper proceeds as follows. In section 2 we prove Theorem \ref{localization}. In section 3 we verify the existence
result of Hultgren in \cite{Hultgren17} by checking the vanishing of $\Fut$ as an application of Theorem \ref{localization}.

\section{Localization Formula}

We first consider an ample line bundle $L\to M$ with $c_1(L)=\frac{1}{2\pi}[\omega]$ where $[\omega]$ is a \ka class of $M$. Let $e_U$ be a non-vanishing local holomorphic section of $L|_U$ where $U$ is an open set of $M$. Then $e_U$ determines a local trivialization of the line bundle $L|_U\cong U\times\C$, given by $ze_U\mapsto (p, z)$, where $z$ is the fiber coordinate. Let $h$ be the Hermitian metric of $L$, and $h_U=h(e_U, e_U)$. The local connection form is given by $\theta_U=\pa\log h_U$. Let 
\begin{equation}
\theta=\theta_U+\frac{dz}{z}, 
\end{equation}
then $\theta$ is a globally defined connection form on the associated principle $C^*$-bundle. To see this, we first remark that $\frac{dz}{z}$ is the Maurer-Cartan form of $C^*$. If $U\cap V\neq\emptyset$, and we take another trivialization on $L|_V\cong V\times\C$, given by $we_V\mapsto (p, w)$, where $e_V$ is a non-vanishing local holomorphic section and $w$ is the fiber coordinate. Let $f$ be the non-vanishing holomorphic function such that $e_V=f e_U$, then $h_V=|f|^2 h_U$ and $z=fw$. Then,
\begin{align*}
\theta_V+\frac{dw}{w}=\pa\log|f|^2h_U+\frac{f}{z}d\Big(\frac{z}{f}\Big)=\frac{df}{f}+\pa\log h_U+\frac{dz}{z}-\frac{df}{f}=\theta_U+\frac{dz}{z}.
\end{align*}
Hence $\theta=\theta_U + \frac{dz}{z}$ is independent of the trivialization. Obviously 
$\sqrt{-1} \barpartial \theta = \omega$.
Let $u$ be a complex-valued smooth function such that 
\begin{equation}\label{Hamiltonian2}
i_X\omega = \bar\pa(\sqrt{-1}u).
\end{equation}
It is well-known (c.f. \cite{FM02} for example) that a Hamiltonian vector field $X$ written in this way 
lifts $L$ uniquely up to $cz \frac{\partial}{\partial z}$ for a constant $c$. 
Let $\tilde{X}$ be a lift of $X$ to $L$. 
Then obviously $u_X := -\theta (\tilde{X})$ is a Hamiltonian function for $X$ and $-\theta(\tilde{X} - cz \frac{\partial}{\partial z}) = u_X + c$. Thus, the ambiguity of $c_\ga$ for $L_\ga$ above appears in this way.
The connection form $\theta$ determines a horizontal lift $X^h$ of $X$, given by 
\[X^h=\tilde{X}-\theta(\tilde{X})z\frac{\pa}{\pa z}.\]
Apparently, this expression is independent of the lift $\tilde{X}$ and $\theta(X^h)=0$.

Now, for each ample line bundle $L_\ga\to M$, $\ga=1, \dots, k$, choose Hermitian metric $h_\ga$, let $\theta_\ga$ be corresponding connection form on the associated principal $C^*$-bundle, and $\Theta_\ga$ is the curvature form such that $\Theta_\ga=\bar\pa\pa\log h_\ga=-\sqrt{-1}\omega_\ga$.

Hence, with a choice of a Hamiltonian function $u_\ga$, the lifted holomorphic vector field $X_\ga$ (omitting the tilde) of $X$ on $L_\ga$ is 
\[X_\ga=X_\ga^h-u_\ga z\frac{\pa}{\pa z}\]
where $X_\ga^h$ is the horizontal lift of $X$. Then of course
\begin{equation*}
u_\ga=-\theta_\ga(X_\ga).
\end{equation*}
\noindent
The infinitesimal action on the space $\Gamma(L_\ga)$ of holomorphic sections of $L_\ga$ is given by 
\begin{align*}
\Lambda_\ga: \Gamma(L_\ga)&\to \Gamma (L_\ga)\\
 s\quad&\mapsto\Lambda_\ga(s)=\nabla_X^\ga s+ u_\ga s
\end{align*}
where $\nabla^\ga$ is the covariant derivative determined by $\theta_\ga$. 

Then we can check that for $f\in C^\infty (M)$, $s\in \Gamma(L_\ga)$, 
\begin{itemize}
\item [(1)] $\Lambda_\ga$ satisfies the Leibniz Rule. 
\begin{align*}
\Lambda_\ga(fs)&=\nabla_X^\ga(fs)+ u_\ga fs\\
&=X(f)s+f\nabla_X^\ga s+ fu_\ga s\\
&=X(f)s+f\Lambda_\ga s.
\end{align*}
\item [(2)] $\bar\pa\Lambda_\ga=\Lambda_\ga\bar\pa$. This follows from
\begin{align*}
\bar\pa\Lambda_\ga s&=\bar\pa(i_X\nabla^\ga s+ u_\ga s)=-i_X\bar\pa\nabla^\ga s+\bar\pa u_\ga s\\
&=\big(-i_X\Theta_\ga+\bar\pa u_\ga\big)s=\sqrt{-1}\big(i_X\omega_\ga-\bar\pa(\sqrt{-1}u_\ga)\big)s=0.
\end{align*}
\item[(3)] It is obvious that $\Lambda_\ga|_{\text{Zero}(X)}=u_\ga|_{\text{Zero}(X)}$ is a linear map on $\Gamma(L_\ga|_{\text{Zero}(X)})$.
\end{itemize}
\noindent
This implies $\Lambda_\ga|_{\text{Zero}(X)}\in End(L_\ga|_{\text{Zero}(X)})$. This endomorphism along the zero set of $X$ can be extended a global endomorphism of $L_\ga$ by letting for $s\in \Gamma(L_\ga)$
\[E_\ga s=\Lambda_\ga s-\nabla_X^\ga s=u_\ga s=-\theta_\ga(X_\ga)s.\]
Then $E_\ga\in End(L_\ga)$ and 
\begin{equation}
\bar\pa E_\ga=\bar\pa u_\ga=i_X(-\sqrt{-1}\omega_\ga)=i_X\Theta_\ga.
\end{equation}

The above discussion enables us to write the Lie algebra character (\ref{futaki}) as
\begin{equation}\label{futall}
\begin{split}
\Fut(X) &= \sum\limits_{\ga=1}^k\frac{\int_M u_\ga\ \omega_\ga^m}{\int_M \omega_\ga^m}\\
 &=\frac{1}{m+1}\sum\limits_{\ga=1}^k\frac{\int_M(u_\ga+\omega_\ga)^{m+1}}{\int_M(u_\ga+\omega_\ga)^m}\\
 &=\frac{1}{m+1}\sum\limits_{\ga=1}^k\frac{\int_M(-\theta_\ga(X_\ga)+\sqrt{-1}\Theta_\ga)^{m+1}}{\int_M(-\theta_\ga(X_\ga)+\sqrt{-1}\Theta_\ga)^m}\\
 &=\frac{1}{m+1}\sum\limits_{\ga=1}^k\frac{\int_M(E_\ga+\sqrt{-1}\Theta_\ga)^{m+1}}{\int_M(E_\ga+\sqrt{-1}\Theta_\ga)^m}.
\end{split}
\end{equation}
\noindent
Here we remark that the both expressions $\int_M(-\theta_\ga(X_\ga)+\sqrt{-1}\Theta_\ga)^{m+1}$ and $\int_M(-\theta_\ga(X_\ga)+\sqrt{-1}\Theta_\ga)^{m}$ are independent of the choice of Hermitian metric $h_\ga$. This could either follow from \cite{futakimorita85} (proposition 2.1) or argue as follows. We choose a family of Hermitian metrics ${h}_\ga(t)$, let ${h}_\ga(t)=e^{-t\varphi_\ga}h_\ga$, for $\varphi_\ga\in C^\infty(M)$. Then 
\[{\theta}_\ga(t)=\pa\log h_\ga(t)+\frac{dz}{z}=\theta_\ga-t\pa\varphi_\ga\] 
is the corresponding family of connections on associated principle $C^*$-bundle, and the curvature forms are
\[\Theta_\ga(t)=\Theta_\ga+t\pa\bar\pa\varphi_\ga,\]
and we compute that
\[i_X\Theta_\ga(t)=i_X\Theta_\ga+i_X(t\pa\bar\pa\varphi_\ga)=\bar\pa\big(u_\ga+tX(\varphi_\ga)\big),\]
we let $ u_\ga(t)=u_\ga+tX(\varphi_\ga)$. This $ u_\ga(t)$ is a Hamiltonian function of $X$ for the K\"ahler form
$\omega_\ga(t)$ corresponding to $h_\ga(t)$. 
As we saw above
%$\theta_\ga(t)$ determines a horizontal lift $ X_\ga^h(t)$ of $X$ by 
%\[ X_\ga^h(t)=\tilde{X}-\theta_{\ga}(t)(\tilde{X})z\frac{\pa}{\pa z},\]
%where $\theta_{\ga, u}(t)$ is the local connection form, and 
the lifted vector field on $L_\ga$ is given by
\[ X_\ga(t)=X_\ga^h(t)- u_\ga(t)z\frac{\pa}{\pa z}.\]
Then 
\[-\theta_\ga(t)(X_\ga(t))= u_\ga(t)=-\theta_\ga(X_\ga)+tX(\varphi_\ga).\]
We will check the metric independence of $\int_M(-\theta_\ga(X_\ga)+\sqrt{-1}\Theta_\ga)^{m+1}$, and similar argument works for $\int_M(-\theta_\ga(X_\ga)+\sqrt{-1}\Theta_\ga)^{m}$. We compute that
\begin{align*}
&\frac{d}{dt}\int_M \big(-\theta_\ga(t)(X_\ga(t))+\sqrt{-1}\Theta_\ga(t)\big)^{m+1}\\
=&(m+1)\int_M \big(-\theta_\ga(t)(X_\ga(t))+\sqrt{-1}\Theta_\ga(t)\big)^m\wedge\big(X(\varphi_\ga)+\sqrt{-1}\pa\bar\pa\varphi_\ga\big)\\
=&(m+1)\Big(\int_M X(\varphi_\ga)(\sqrt{-1}\Theta_\ga(t))^{m}-m\theta_\ga(t)(X_\ga(t))(\sqrt{-1}\Theta_\ga(t))^{m-1}\wedge\sqrt{-1}\pa\bar\pa\varphi_\ga\Big)\\
=&(m+1)\Big(\int_M X(\varphi_\ga)(\sqrt{-1}\Theta_\ga(t))^m-m\int_M\bar\pa\big(\theta_\ga(t)(X_\ga(t))\big)(\sqrt{-1}\Theta_\ga(t))^{m-1}\wedge\sqrt{-1}\pa\varphi_\ga\Big)\\
=&(m+1)\Big(\int_M X(\varphi_\ga)(\sqrt{-1}\Theta_\ga(t))^m+m\int_M i_X\Theta_\ga(t)\wedge(\sqrt{-1}\Theta_\ga(t))^{m-1}\wedge\sqrt{-1}\pa\varphi_\ga\Big)\\=&0.
\end{align*}
%%%%%%%%
\begin{proof}[Proof of Theorem \ref{localization}]
Now, we follow an argument in the book \cite{futaki88} (see Theorem 5.2.8), originally due to Bott \cite{bott67} to give the localization formula.

Consider an invariant polynomial $P$ of degree $(m+l)$ for $l=0, 1$, let
\[P_\ga(E_\ga+\sqrt{-1}\Theta_\ga)=\sum\limits_{r=0}^{m+l}P_{\ga, r}(E_\ga, \sqrt{-1}\Theta_\ga),\]
where 
\[P_{\ga, r}(E_\ga, \sqrt{-1}\Theta_\ga)=\binom {m+l}{r}P(E_\ga, \dots, E_\ga;\underbrace{\sqrt{-1}\Theta_\ga, \dots, \sqrt{-1}\Theta_\ga}_r ).\]
Since $\bar\pa E_\ga=i_X\Theta_\ga$, we have 
\[\sqrt{-1}\,\bar\pa P_\ga=i_X P_\ga.\]

Define a $(1, 0)$ form $\pi_\ga$ as follows: for a holomorphic vector field $Y$, 
\[i_Y\pi_\ga=\frac{\omega_\ga(Y, X)}{\omega_\ga(X, X)},\]
then 
\[i_X\pi_\ga=1,\quad \text{ and }\quad i_X\bar\pa\pi_\ga=0.\]
We further define 
\[\eta_\ga=\frac{\pi_\ga}{1-\sqrt{-1}\bar\pa\pi_\ga}\wedge P_\ga(E_\ga+\sqrt{-1}\Theta_\ga),\]
then $\eta_\ga$ is defined outside zero sets of $X$. The computation shows 
\[P_\ga(E_\ga+\sqrt{-1}\Theta_\ga)=-\sqrt{-1}\,\bar\pa\eta_\ga+i_X\eta_\ga.\]
Let $B_\epsilon(Z)$ be an $\epsilon$-neighbourhood of $Z$. Then 
\begin{align*}
&\int_M P_\ga(E_\ga+\sqrt{-1}\Theta_\ga)\\
=&\lim\limits_{\epsilon\to 0}\int_{M-B_{\epsilon(Z)}} P_\ga(E_\ga+\sqrt{-1}\Theta_\ga)\\
=&\sqrt{-1}\lim\limits_{\epsilon\to 0}\int_{M-B_{\epsilon(Z)}}-\bar\pa\eta_\ga^{(2m-1)}=\sqrt{-1}\sum\limits_{\lambda\in \Lambda}\lim\limits_{\epsilon\to 0}\int_{\pa B_{{\epsilon}(Z)}} \eta_\ga^{(2m-1)}\\
=&\sqrt{-1}\sum\limits_{\lambda\in \Lambda}\lim\limits_{\epsilon\to 0}\int_{\pa B_{{\epsilon}(Z)}} \pi_\ga\wedge\big(1+(\sqrt{-1}\bar\pa\pi_\ga)+(\sqrt{-1}\bar\pa\pi_\ga)^2+\dots+(\sqrt{-1}\bar\pa\pi_\ga)^{m-1}\big)\\
&\qquad\qquad\qquad\qquad\qquad\qquad \wedge \sum\limits_{r=0}^{m-1} P_{\ga, r}(E_\ga, \sqrt{-1}\Theta_\ga).
\end{align*}
As computed in Theorem 5.2.8 in \cite{futaki88} or \cite{bott67} , 
\[(2\pi)^{-m}\int_M P_\ga(E_\ga+\sqrt{-1}\Theta_\ga)=\sum\limits_{\lambda\in\Lambda}\int_{Z_\lambda}\frac{P_\ga(E_\ga+\sqrt{-1}\Theta_\ga)|_{Z_\lambda}}{\det\big((2\pi)^{-1}(L^{N_\ga}(X)+\sqrt{-1}K_\ga)\big)},\]
where $K_\ga$ is the curvature of the normal bundle $N_\ga$ with respect to the induced metric. 
Taking $P=tr^{m+1}$ and $P=tr^m$, and apply above to (\ref{futall}), we obtain the localization formula of $\Fut(X)$ in the coupled case (\ref{locfut}).
\end{proof}

\section{Application of Localization Formula}

Before computing the example, we remark that by Theorem 3.2 in \cite{futakizhang}, \eqref{Th3.3} is equivalent to
$$ \int_M (u_1 + \cdots + u_k) dV = 0$$
where $dV = e^{f_\ga} \omega_\ga^m$ which is independent of $\ga$ by the normalization \eqref{normalization}. 
By Theorem 5.2 in \cite{futakizhang} this condition is equivalent to 
\begin{equation}\label{Minkowski}
\sum_{\ga=1}^k \mathcal P_\ga = \mathcal P_{-K_M}
\end{equation}
where $\mathcal P_\ga$ is the moment map image of $\omega_\ga$.

We consider the tautological line bundles $\mathcal{O}_{\mathbb{CP}^1}(-1)\to \mathbb{CP}^1$ and $\mathcal{O}_{\mathbb{CP}^2}(-1)\to \mathbb{CP}^2$, and the bundle $E=\mathcal{O}_{\mathbb{CP}^1}(-1)\oplus\mathcal{O}_{\mathbb{CP}^2}(-1)$ over $ \mathbb{CP}^1\times\mathbb{CP}^2$. Let $M$ be the total space of the projective line bundle $\mathbb{P}(E)$ over $\mathbb{CP}^1\times\mathbb{CP}^2$. In local coordinates, we let
\begin{align*}
\mathbb{CP}^1=\{(b_0: b_1)\},\qquad \mathbb{CP}^2=\{(a_0:a_1:a_2)\},
\end{align*}
\begin{align*}
\mathcal{O}_{\mathbb{CP}^1}(-1)&=\{[(w_0, w_1),(b_0:b_1)]|\ (w_0, w_1)=\lambda(b_0, b_1)\text{\ for\ some\ }\lambda\in\mathbb C\},\\
\mathcal{O}_{\mathbb{CP}^2}(-1)&=\{[(z_0, z_1, z_2), (a_0:a_1:a_2)]|\ (z_0, z_1, z_2)=\mu(a_0, a_1, a_2))\text{\ for\ some\ }\mu\in \mathbb C\},
\end{align*}
\begin{align*}
M=&\{[(z_0: z_1:z_2: w_0:w_1), (a_0: a_1:a_2), (b_0: b_1)]|\\
&\ (w_0, w_1)=\lambda(b_0, b_1), (z_0, z_1, z_2)=\mu(a_0, a_1, a_2))\text{\ for\ some\ } (\lambda,\mu) \ne (0,0)\text{\ in\ }\mathbb C \times \mathbb C\}.
\end{align*}
The $(\CC^*)^4$-action on $M$ is defined by extending the $\CC^*$-action on $\mathbb{CP}^1$ and $(\CC^*)^2$-action on $\mathbb{CP}^2$. We let $(t_1, t_2, t_3, t_4)\in(\CC^*)^4$, then 
\begin{align*}
&(t_1, t_2, t_3, t_4)\cdot[(z_0: z_1:z_2: w_0:w_1), (a_0: a_1:a_2), (b_0: b_1)]\\
=&[(z_0: t_1z_1:t_2z_2: t_4w_0:t_4t_3w_1), (a_0: t_1a_1:t_2a_2), (b_0:t_3b_1)].
\end{align*}
There are totally seven $(\CC^*)^4$-invariant divisors;
%\begin{enumerate}
%\item 
\begin{align*}
D_1=\{z_0=a_0=0\}, \quad D_2=\{z_1=a_1=0\}, \quad D_3=\{z_2=a_2=0\},
\end{align*}
%\noindent
which are identified with $\mathbb{CP}^1$-bundle over $\mathbb{CP}^1\times\mathbb{CP}^1$;
%\item
\begin{align*}
D_4=\{b_0=w_0=0\},\quad D_5=\{b_1=w_1=0\}, 
\end{align*}
%\noindent
which are identified with $\mathbb{CP}^1$-bundle over $\mathbb{CP}^2$;
%\item
\begin{align*}
D_6=\{z_0=z_1=z_2=0\}, \quad D_7=\{w_0=w_1=0\},
\end{align*}
%\noindent
which are identified with $\mathbb{CP}^1\times\mathbb{CP}^2$.
%\end{enumerate}
It is known that
\[K_M^{-1}=\sum\limits_{i=1}^7 D_i.\]
As in \cite{Hultgren17}, we consider the following decomposition for $c\in(1/4, 3/4)$ which is ampleness condition
for the line bundles associated with $D(c)$ and $D(1-c)$ below.  Define
\begin{align*}
D(c)&=\frac{1}{2}K_M^{-1}+(c-\frac{1}{2})(D_4+D_5)\\
 D(1-c)&=\frac{1}{2}K_M^{-1}+(\frac{1}{2}-c)(D_4+D_5),
\end{align*}
then
\begin{equation}\label{decomp}
K_M^{-1}=D(c)+D(1-c).
\end{equation}
We remark that the torus action preserves above decomposition. 

Note also that the invariant $\Fut$ is invariant under any automorphism of $M$ preserving the decomposition \eqref{decomp}. 
Using the automorphism $(b_0,b_1) \mapsto (b_1,b_0)$ one can see $\Fut(X_3) = \Fut(-X_3)$ and thus
$\Fut(X_3) = 0$ for the infinitesimal generator $X_3$ for the $t_3$-action, and similarly $\Fut(X_1) = \Fut(X_2) = 0$ for the infinitesimal generators $X_1$ and $X_2$ of $t_1$ and $t_2$-actions using the automorphisms induced by the odd permutations of the coordinates $(a_0:a_1:a_2)$. Hence, to compute the coupled $\Fut$ invariant, it is sufficient to consider the action of one parameter subgroup $(1, 1, 1, t_4)$ on $M$ by 
\begin{align*}
&(1, 1, 1, t_4)\cdot[(z_0: z_1:z_2: w_0:w_1), (a_0: a_1:a_2), (b_0: b_1)]\\
=&[(z_0: z_1:z_2: t_4w_0:t_4w_1), (a_0: a_1:a_2), (b_0:b_1)].
\end{align*}
For this action, let $\xi=\lambda/\mu$, $\eta=1/\xi$, then the associated holomorphic vector field is
\[X=\xi\frac{\pa }{\pa \xi}=-\eta\frac{\pa }{\pa \eta}.\]
Zero sets are
\begin{align*}
Z_\infty=\{\mu=0\}=D_6, \quad\text{ and }\quad Z_0=\{\lambda=0\}=D_7.
\end{align*}
Since
\begin{align*}
\mathbb{P}(\mathcal{O}_{\mathbb{CP}^1}(-1)\oplus\mathcal{O}_{\mathbb{CP}^2}(-1))&=\mathbb{P}\Big((\mathcal{O}_{\mathbb{CP}^1}(-1)\oplus\mathcal{O}_{\mathbb{CP}^2}(-1))\otimes\mathcal{O}_{\mathbb{CP}^2}(1)\Big)\\
&=\mathbb{P}\Big((\mathcal{O}_{\mathbb{CP}^1}(-1)\otimes\mathcal{O}_{\mathbb{CP}^2}(1))\oplus\mathcal{O}_{\mathbb{CP}^2}\Big),
\end{align*}
the normal bundle of $Z_\infty$ is
\[\nu(Z_\infty)=\mathcal{O}_{\mathbb{CP}^1}(-1)\otimes\mathcal{O}_{\mathbb{CP}^2}(1),\]
similarly, the normal bundle of $Z_0$ is
\[\nu(Z_0)=\mathcal{O}_{\mathbb{CP}^1}(1)\otimes\mathcal{O}_{\mathbb{CP}^2}(-1)=\nu(Z_\infty)^{-1}.\]

Let $a, b$ be the positive generators of $H^2(\mathbb{CP}^1, \ZZ)$ and $H^2(\mathbb{CP}^2, \ZZ)$. Then 
\[c_1(\mathbb{CP}^1)=2a, \quad c_1(\mathbb{CP}^2)=3b,\]
and
\[c_1(K_M^{-1})|_{Z_\infty}=c_1(Z_\infty)+c_1(\nu(Z_\infty))=2a+3b-a+b=a+4b.\]
Similarly we have
\[c_1(K_M^{-1})|_{Z_0}=3a+2b.\]
Since the line bundle $[D_4]$ restricted to $Z_\infty = D_6$ is isomorphic to the line bundle corresponding to
the divisor $\{b_0 = 0\}$ in $\mathbb{CP}^1\times\mathbb{CP}^2$ we have $c_1([D_4])|_{Z_\infty} = a$.
Similarly we have 
\[c_1([D_4])_{Z_0}=c_1([D_5])|_{Z_\infty}=c_1([D_5])_{Z_0}=a.\]
Then
\begin{align*}
c_1(D(c))|_{Z_\infty}=\frac{1}{2}(a+4b)+(c-\frac{1}{2})2a=(2c-\frac{1}{2})a+2b,\\
c_1(D(c))|_{Z_0}=\frac{1}{2}(3a+2b)+(c-\frac{1}{2})2a=(2c+\frac{1}{2})a+b.
\end{align*}
%Let $u$ be the Hamiltonian function of vector field $\frac{\sqrt{-1}}{2\pi}X$, it is known that the value of $u$ on the zero sets of $X$ is a topological invariant. By the Atiyah's convexity theorem of moment polytope, we can determine 
To see the value of $u$ along the zero sets of $X$ we may use the description of the moment polytope $P(c)$ in 
\cite{Hultgren17}
$$ P(c) = \{ y \in \mathbb R^4 : \langle y, d_i \rangle \le \frac12, i \ne 4,5,\ \langle y, d_i\rangle \le c, i = 4,5\}$$
where $d_i$ are as described in \cite{Hultgren17}. 
Since $P(c) + P(1-c) = P_{-K_M}$ the moment polytopes are those obtained by the Hamiltonian functions satisfying \eqref{Th3.3} 
as follows from the arguments of the beginning of this section.
From this description for $d_6 = (0,0,0,-1)$ and 
$d_7=(0,0,0,1)$ we see
\[u\big|_{Z_\infty}=-\frac{1}{2}, \qquad u\big|_{Z_0}=\frac{1}{2}.\] 
By using the fact 
\[a^2=b^3=0,\]
we first compute
\begin{equation}\label{volc}
\begin{split}
\Vol(D(c))&=\Big[\frac{(u|_{Z_\infty}+c_1(D(c))|_{Z_\infty})^4}{u|_{Z_\infty}+c_1(\nu(Z_\infty))}+\frac{(u|_{Z_0}+c_1(D(c))|_{Z_0})^4}{u|_{Z_0}+c_1(\nu(Z_0))}\Big][\mathbb{CP}^1\times\mathbb{CP}^2]\\
&=\Big[\frac{\big(-1/2+(2c-1/2)a+2b\Big)^4}{-1/2-a+b}+\frac{\Big(1/2+(2c+1/2)a+b\big)^4}{1/2+a-b}\Big][\mathbb{CP}^1\times\mathbb{CP}^2]\\
&=112c-6,%=2(56c-3),
\end{split}
\end{equation}
replacing $c$ by $1-c$, we get
\begin{equation}\label{vol1c}
\Vol(D(1-c))=106-112c.%=2(56(1-c)-3).
\end{equation}

We also need to compute the numerators in the localization formula. For the divisor $D(c)$, 
\begin{equation}\label{numc}
\begin{split}
&\Big[\frac{(u|_{Z_\infty}+c_1(D(c))|_{Z_\infty})^5}{u|_{Z_\infty}+c_1(\nu(Z_\infty))}+\frac{(u|_{Z_0}+c_1(D(c))|_{Z_0})^5}{u|_{Z_0}+c_1(\nu(Z_0))}\Big][\mathbb{CP}^1\times\mathbb{CP}^2]\\
=&\Big[\frac{\big(-1/2+(2c-1/2)a+2b\Big)^5}{-1/2-a+b}+\frac{\Big(1/2+(2c+1/2)a+b\big)^5}{1/2+a-b}\Big][\mathbb{CP}^1\times\mathbb{CP}^2]\\
=&-30c+12,%=6(-5c+2),
\end{split}
\end{equation}
replacing $c$ by $1-c$, we get for divisor $D(1-c)$,
\begin{equation}\label{num1c}
\begin{split}
&\Big[\frac{(u|_{Z_\infty}+c_1(D(1-c))|_{Z_\infty})^5}{u|_{Z_\infty}+c_1(\nu(Z_\infty))}+\frac{(u|_{Z_0}+c_1(D(1-c))|_{Z_0})^5}{u|_{Z_0}+c_1(\nu(Z_0))}\Big][\mathbb{CP}^1\times\mathbb{CP}^2]\\=&30c-18.%=6(-5(1-c)+2).
\end{split}
\end{equation}

Plug  above (\ref{volc}), (\ref{vol1c}), {\ref{numc}), (\ref{num1c}) into the localization formula  (Theorem \ref{localization}),  we obtain
\begin{equation}
\begin{split}
\Fut(X)=&\frac{\Big[\frac{(u|_{Z_\infty}+c_1(D(c))|_{Z_\infty})^5}{u|_{Z_\infty}+c_1(\nu(Z_\infty))}+\frac{(u|_{Z_0}+c_1(D(c))|_{Z_0})^5}{u|_{Z_0}+c_1(\nu(Z_0))}\Big][\mathbb{CP}^1\times\mathbb{CP}^2]}{\Vol(D(c))}\\&+\frac{\Big[\frac{(u|_{Z_\infty}+c_1(D(1-c))|_{Z_\infty})^5}{u|_{Z_\infty}+c_1(\nu(Z_\infty))}+\frac{(u|_{Z_0}+c_1(D(1-c))|_{Z_0})^5}{u|_{Z_0}+c_1(\nu(Z_0))}\Big][\mathbb{CP}^1\times\mathbb{CP}^2]}{\Vol(D(1-c))}\\
=&\frac{-30c+12}{112c-6}+\frac{30c-18}{106-112c}\\
=&\frac{-15(112c^2-112c+23)}{(56c-3)(56c-53)},
\end{split}
\end{equation}
therefore, the invariant $\Fut$ character vanishes when 
\[c=\frac{1}{2}\pm\frac{1}{4}\sqrt{\frac{5}{7}}.\]
This is the same as in \cite{Hultgren17}.

\end{document}